\documentclass{amsart}

\usepackage{graphicx}
\usepackage{subcaption}
\usepackage{diagbox}
\usepackage{multirow}
\usepackage{a4wide}

\newcommand{\erf}{\text{erf}}
\newcommand{\F}{\mathcal{F}}
\newcommand{\Ec}{\mathcal{E}}

\newcommand{\e}{\epsilon}

\newcommand{\test}{\varphi}
\renewcommand{\div}{{\rm div}}
\newcommand{\E}{\mathbb{E}}
\renewcommand{\P}{\mathcal{P}}
\newcommand{\C}{\mathcal{C}}
\newcommand{\rr}{\mathbb{R}}
\newcommand{\nn}{\mathbb{N}}
\renewcommand{\rho}{\varrho}

\newcommand{\X}{\bar{X}}
\newcommand{\tol}{\text{tol}}
\newcommand{\ldb}{\mathopen{\ooalign{\makebox[.4em][l]{$\lbrack$}\cr\makebox[.4em][r]{$\lbrack$}\cr}}}                                                       
\newcommand{\rdb}{\mathclose{\ooalign{\makebox[.4em][l]{$\rbrack$}\cr\makebox[.4em][r]{$\rbrack$}\cr}}}
\providecommand{\norm}[1]{\lVert#1\rVert}

\newtheorem{proposition}{Proposition}[section]
\newtheorem{lemma}{Lemma}[section]

\theoremstyle{definition}
\newtheorem{definition}{Definition}[section]
\newtheorem{remark}{Remark}[section]

\title{A consensus-based model for global optimization and its mean-field limit}

\author[Pinnau]{Ren\'e Pinnau}
\address[Ren\'e Pinnau]{\newline Department of Mathematics, Technische Universit\"at Kaiserslautern, 
	\newline Erwin-Schr\"odinger-Strasse, 67663 Kaiserslautern, Germany}
\email{pinnau@mathematik.uni-kl.de}

\author[Totzeck]{Claudia Totzeck}
\address[Claudia Totzeck]{\newline Department of Mathematics, Technische Universit\"at Kaiserslautern, 
	\newline Erwin-Schr\"odinger-Strasse, 67663 Kaiserslautern, Germany}
\email{totzeck@mathematik.uni-kl.de}

\author[Tse]{Oliver Tse}
\address[Oliver Tse]{\newline Department of Mathematics and Computer Science, Eindhoven University of Technology, 
	\newline P.O. Box 513, 5600MB Eindhoven, The Netherlands}
\email{o.t.c.tse@tue.nl}

\author[Martin]{Stephan Martin}
\address[Stephan Martin]{\newline SAP Walldorf, Germany}
\email{smartin.research@gmail.com}

\begin{document}

\maketitle

\begin{abstract}
We introduce a novel first-order stochastic swarm intelligence (SI) model in the spirit of consensus formation models, namely a consensus-based optimization (CBO) algorithm, which may be used for the global optimization of a function in multiple dimensions. The CBO algorithm allows for passage to the mean-field limit, which results in a nonstandard, nonlocal, degenerate parabolic partial differential equation (PDE). Exploiting tools from PDE analysis we provide convergence results that help to understand the asymptotic behavior of the SI model. We further present numerical investigations underlining the feasibility of our approach.\\
\end{abstract}

\keywords{\small Keywords: Consensus formation; global optimization; interacting system; mean-field limit; stochastic differential equations.\\}

\subjclass{\small AMS Subject Classification: 34F05, 35B40, 90C26}

\section{Introduction}
Many applied problems rely on the numerical optimization of a target function. While the problem of local minimization is well understood, the numerical identification of global minima is still a challenging task \cite{LoSch13}.
Over the last decades, {\em metaheuristics} have played an increasing role in the design of fast algorithms to provide sufficiently good solutions to an optimization problem. {\em Swarm intelligence} (SI), for one, is a class of metaheuristic algorithms that is mostly inspired by nature, especially biological systems \cite{parsopoulos2002recent,poli2007particle}. SI systems consists typically of a population of simple agents interacting with one another and with their environment. These interactions often lead to the emergence of an 'intelligent' collective behavior, unknown to the individual agents. The main advantage of these methods over other global minimization strategies such as {\em simulated annealing} is their resilience to the problem of saturating at a local minima.

Notable algorithms within this class include {\em particle swarm optimization} (PSO) (see, e.g. \cite{kennedy2010particle,poli2007particle} and the references therein), {\em ant colony optimization} {mohan2012survey} (ACO) and {\em artificial bee colony} optimization \cite{karaboga2014comprehensive} (ABC). Their basic set-up is a population of agents exploring a bounded domain and evaluating the function along each trajectory. The force driving the agents' motion is derived from a mixture of individual steering and communication with the collective. For instance, agents can store the position of the best function value along their individual path on the one side, and get information on the global best location by communicating with the others. The general difficulty lies in the development of robust algorithms with a well-balanced exploration and exploitation ability. A vast number of these algorithms have been suggested in the literature and the variants differ with respect to memory effects, stochasticity, time discretization and other features.

On the other hand, individual-based models are also frequently used to investigate collective behavior effects in applications such as mathematical biology \cite{naldi2010mathematical}, swarming \cite{CarrilloCS,cucker1,vicsek1995novel}, crowd dynamics \cite{bellomo2011modeling,bellomo2016behavioral,bellomo2012modeling} or opinion formation \cite{galam2012sociophysics,helbing2010quantitative,motschtadmor}. Ref.~\cite{naldi2010mathematical}, for example, provides a series of articles that demonstrates the common methodological approaches and tools for modeling and simulating collective behavior. As for opinion dynamics, Ref.~\cite{helbing2010quantitative} provides a unified and comprehensive overview of the different stochastic methods, their interrelations and properties. In particular, psychologists have begun investigating the proponents of change in attitude and opinion using mathematical models as early as the 1950's \cite{festinger1962theory,french1956formal,osgood1955principle}. Opinion formation is one of the fundamental aspects of self-organization in networks of systems \cite{dolfin2015modeling,jia2015opinion,knopoff2014mathematical,pareschi2013interacting}.

In general, opinion dynamics within an interacting population can lead to either consensus, polarisation or even fragmentation. A thorough understanding of such phenomena would, initially, require the formulation of mathematical models, which describe the evolution of opinions in the population under investigation. An early formulation of such a model was given by French in \cite{french1956formal} to understand complex phenomena found empirically about groups. A well recognised linear model that leads to consensus is the de Groot's model \cite{degroot1974reaching} that was analysed in the work \cite{chatterjee1977towards}. Later on, Vicsek et.~al.~introduced a simple local averaging rule as a basis for studying the cooperative behaviour of animals \cite{vicsek1995novel}. A simplification of the Vicsek model that exhibits interesting clustering and consensus effects is found in \cite{hegselmann2002opinion,krause2000discrete}, which makes it more attractive for engineering implementations. Other notable articles on opinion dynamics include \cite{acemoglu2011opinion,galam1991towards}.

Unfortunately, a through investigation of individual-based models become exceedingly difficult, and sometimes even impossible, for very large numbers of individuals. A very popular workaround is to employ techniques from kinetic theory to analyze such models by the means of partial differential equations (PDE), which describe the evolution of the probability distribution of agents, e.g., by passing to mean field equations \cite{CarrilloCanizoBolley,CarrilloCS,dobrushin1989dynamical,HaTadmor,Sznitman} or by means of evolutionary differential games \cite{marsan2016stochastic,aletti2007first,bellomo2016mathematics,boudin2009kinetic}. Stationary states, ergodicity and pattern formations may then be 'easily' investigated on the continuous PDE level rather than on the discrete particle system. 

In this work we introduce a stochastic SI algorithm called {\em consensus-based optimization} (CBO) that bears strong resemblance to consensus formation models. In such models, agents adapt their opinion, revise their beliefs, or change their behavior as a result of social interactions with other agents, thereby leading to either consensus, polarization or fragmentation within an interacting population \cite{burini2016collective,krause2000discrete,naldi2010mathematical}. A metaheuristic based on a discrete consensus model may be found in \cite{Sichani13}.

The objective of this paper is twofold: 
\begin{enumerate}
\item[(i)] Introduce an agent-based global optimization strategy that is efficient and robust and that further allows passage to the mean-field limit.

\item[(ii)] Partially justify the efficiency of the method by analyzing its mean-field limit equations.
\end{enumerate}

The study of SI algorithms from a mean-field point of view has, to the best of our knowledge, not been fully explored before. The mean-field perspective will allow to get a deeper understanding of the performance of the agent-based algorithm, especially in regards to convergence properties.

This work is organized as follows: We begin Section~\ref{sec:model} with the formulation of the CBO algorithm. Here, we also derive the mean-field limit equation and draw conclusions on the analytic properties of the algorithm. In Section~\ref{sec:numerics}, we present and justify results on the numerical performance of the algorithm in dimension $d=1$. Additionally, we present numerical results indicating the potential of the CBO algorithm in multiple dimensions in Section~\ref{sec:multi}. In particular, we consider optimizations problems in dimension $d=20$. Finally, concluding remarks are given in Section~\ref{sec:conclusions}. 

\section{A Consensus-Based Optimization (CBO) Algorithm}
\label{sec:model}
In this section we introduce a first-order consensus-based minimization algorithm with smooth deterministic and multiplicative stochastic forcing. The task at hand is to find a global minimum 
\begin{equation}\label{eq:minprob}
 \min\nolimits_{x\in\rr^d} f(x)
\end{equation}
of a given continuous objective function $f\in\C_b(\rr^d;\rr)$ which is assumed to be non-negative and bounded. Here, the number of degrees of freedom for the optimization is given by the spatial dimension $d \in \nn$.

To find this minimum we consider a system of $N\in\nn$ interacting agents with position vector $X_t^i\in\rr^d$, $i = 1, \ldots, N$, which evolves in time with respect to the stochastic differential equations (SDEs) given by
\begin{subequations}\label{eq:particle}
\begin{align}
 dX_t^i = -\lambda (X_t^i-v_f)\,H^\e(f(X_t^i)-f(v_f))\,d t + \sqrt{2} \sigma |X^i_t-v_f| d W_t^i,  \label{model-dynamics}
\end{align}
with drift parameter $\lambda>0$ and noise parameter $\sigma\ge 0$. The point $v_f$ is calculated from the mean
\begin{align}
 v_f  =  \frac{1}{\sum_i \omega_f^\alpha(X^i_t)} \sum\nolimits_i X_t^i\, \omega_f^\alpha(X^i_t). \label{model-consensus}
\end{align}
\end{subequations}
The system is supplemented with an initial condition $X_t^i(0) = X_0^i \in \rr^d$, $i = 1, \ldots, N$, and the
weight function $\omega_f^\alpha$ is given as a power of the reciprocal function value. In this paper, we consider the particular form 
\begin{equation}
\omega_f^\alpha(x) =\exp(-\alpha f(x)), \quad \alpha >0. \label{model-weight}
\end{equation}
Here, $H^\e\colon\rr\to\rr$ denotes a smooth regularization of the Heaviside function.

\begin{remark}
Compared to existing particle optimization algorithms (see e.g. \cite{Sichani13,eberhart1995new,kennedy2010particle,parsopoulos2002recent,poli2007particle}) we avoid the evaluation of $\arg\min_{i=1,\ldots,N} f(X_t^i)$ by considering \eqref{model-consensus} and \eqref{model-weight} instead. Additionally, we incorporate a multiplicative noise term in \eqref{model-dynamics}.
\end{remark}

The positional change of an agent is given by two components: First, each agent compares the function value at its own location $f(X^i_t)$ with the function value at the weighted average location of the collective $f(v_f)$. It is driven towards $v_f$ if $f(X_t^i)-f(v_f)$ is positive, which is determined by $H^\e$. The magnitude of the attraction towards $v_f$ is given by the distance $\lambda|X^i_t-v_f|$, $\lambda>0$, hence agents far away from $v_f$ are strongly attracted, whereas agents located at $v_f$ tend to keep their position.

The second component is a random search term, which we model as independent Brownian motions $\sqrt{2} \sigma W^i_t$ with uniform diffusion parameter $\sigma>0$, whose individual variances are scaled with the distance of the agent from $v_f$. This implies that agents away from $v_f$ exhibit a large noise in their search path, enabling them to explore their current area, while agents near $v_f$ display low or no randomness emphasizing their current position. The current weighted average $v_f$ is determined by a weighted linear interaction process of the collective, where the weight $\omega_f^\alpha$ is a function of each agent's location and the objective function $f$.

\begin{remark}
The model \eqref{eq:particle} neglects the inertia of particles as well as memory effects \cite{kennedy2010particle}. We emphasize that both components of the agent's behavior, attraction towards the average location and random walks at the current location, are scaled with the distance towards $v_f$. Despite its simplicity, we will argue that the algorithm performs surprisingly well and allows for an analytical treatment as the system possesses a mean-field limit towards a partial differential equation (cf.~\cite{CarrilloChoiTotzeckTse}). 
\end{remark}

\subsection{The role of the weighted average $v_f$ and weight function $\omega_f^\alpha$} In most swarm intelligence models, such as PSO, $v_f$ is typically given by the current global best, i.e., 
\[
 v_f=\arg\min\nolimits_{i=1,\ldots,N} f(X_t^i).
\]
However, the existence of a distinguished agent does not permit the passage to the mean-field limit. This motivated the use of a weighted average as given in \eqref{model-consensus}. 

By assuming independence of the processes $X_t^i$, $i=1,\ldots,N$, we may formally pass to the limit $N\to\infty$ in \eqref{model-consensus} to obtain
\begin{align}\label{eq:v_mean-field}
 \frac{1}{\sum_i \omega_f^\alpha(X^i_t)} \sum\nolimits_i X_t^i\, \omega_f^\alpha(X^i_t) \quad\longrightarrow \quad \frac{1}{\int_{\rr^d} \omega_f^\alpha\,d\rho_t} \int_{\rr^d} x\,\omega_f^\alpha\,d\rho_t,
\end{align}
in the distributional sense, which holds due to the law of large numbers. Here $\rho_t\in\P(\rr^d)$ is a Borel probability measure describing the one-particle mean-field distribution (cf.~Section~\ref{sec:mean-field}), which is assumed to be absolutely continuous w.r.t.~the Lebesgue measure $dx$. In this case, $\omega_f^\alpha \rho_t$ satisfies
\begin{align}\label{eq:laplace}
 \lim\nolimits_{\alpha\to\infty}\left(-\frac{1}{\alpha}\log\left(\int_{\rr^d} e^{-\alpha f} d\rho_t\right)\right) = \inf f,
\end{align}
by the {\em Laplace principle} \cite{dembo2009large}. Therefore, if $f$ attains a single minima $x_*\in\text{supp}(\rho_t)$, then the Gibbs-type measure $\eta_t^\alpha:=\omega_f^\alpha \rho_t/\|\omega_f^\alpha\|_{L^1(\rho_t)}\in\P(\rr^d)$ approximates a Dirac distribution $\delta_{x_*}$ at $x_*\in\rr^d$ for sufficiently large $\alpha\gg 1$. In this case, the value on the right-hand side of \eqref{eq:v_mean-field} provides a good estimate of $x_*$. These kind of weighted measures also appear in other metaheuristics, such as simulated annealing \cite{SimAnnCerny,SimAnnKirk}.

For completeness, we include the proof of the convergence mentioned in \eqref{eq:laplace} for sufficiently regular functions $f$. In the rest of this paper, we denote $\P_p^{ac}(\rr^d)$ as the space of Borel probability measures with finite $p$-th moment, which are further absolutely continuous w.r.t.~the Lebesgue measure $dx$ on $\rr^d$.

\begin{proposition}\label{prop:laplace}
 Assume that $f\in\C_b(\rr^d;\rr)$, $f\ge 0$, attains a unique global minimum at the point $x_*\in\rr^d$ and let $\rho\in\P^{ac}(\rr^d)$. Then, we have 
 \begin{align}\label{eq:laplace2}
 \lim\nolimits_{\alpha\to\infty}\left(-\frac{1}{\alpha}\log\left(\int_{\rr^d} e^{-\alpha f} d\rho\right)\right) = f(x_*).
\end{align}
 
\end{proposition}
\begin{proof}
  Let the global minimum be attained at $x_*\in\text{supp}(\rho)$ with $f_*:=f(x_*)$. Notice that $\eta^\alpha=\omega_f^\alpha \rho/\|\omega_f^\alpha\|_{L^1(\rho)}\in\P^{ac}(\rr^d)$. We begin by showing that the functional 
  \[
   \Ec_\alpha(f):=\int_{\rr^d} f\,d\eta^\alpha,
  \]
  converges towards $f_*$ for $\alpha\to\infty$. For this reason, we consider the derivative of $\Ec_\alpha(f-f_*)$ w.r.t.~$\alpha$. Simple calculations lead to
  \[
   \frac{d}{d\alpha} \Ec_\alpha(f-f_*) = -\frac{1}{2}\iint_{\rr^d\times\rr^d} |f(x)-f(y)|^2 d\eta^\alpha(x)\,d\eta^\alpha(y)<0,
  \]
  where the strict inequality holds due to the fact that $f$ is not globally constant. Since $\Ec_\alpha(f)\ge f_*$ for all $\alpha\ge 0$, we see from the differential inequality that $\Ec_\alpha(f)$ decays towards $f_*$ as $\alpha\to\infty$.
  
  Now, let $\e>0$ be arbitrary. From Chebyshev's inequality we further obtain
  \[
   \eta^\alpha(\{x\in\rr^d\,|\, f(x)-f_*\ge \e\}) \le \frac{1}{\e}\int_{\{f-f_*\ge \e\}} (f-f_*)\,d\eta^\alpha \le \frac{1}{\e}\Ec_\alpha(f-f_*).
  \]
  Since the right-hand side tends to zero for $\alpha\to\infty$, we deduce that the measure $\eta^\alpha$ converges in distribution towards the Dirac distribution $\delta_{x_*}$ at the global minima $x_*$, as asserted.
\end{proof}

\begin{remark}
 Notice that since
 \[
  \frac{d}{d\alpha} \log\left(\int_{\rr^d} e^{-\alpha f} d\rho \right) = -\Ec_\alpha(f),
 \]
 we may integrate over $\alpha\ge 0$ and formally pass to the limit $\alpha\to\infty$ to obtain
 \[
  \lim\nolimits_{\alpha\to\infty}\left(-\frac{1}{\alpha}\log\left(\int_{\rr^d} e^{-{\alpha} f} d\rho \right)\right) = \lim\nolimits_{\alpha\to\infty}\frac{1}{\alpha}\int_0^{\alpha} \Ec_{\beta}(f)\,d\beta = f_*,
 \]
 which is precisely the Laplace principle \eqref{eq:laplace}.
\end{remark}

\subsection{The deterministic case $\sigma=0$} 
The interacting particle system \eqref{eq:particle} may also be considered in the deterministic setting $\sigma=0$. According to \eqref{model-consensus}, the consensus location $v_f$ is a convex combination of the agent's positions. It is hence straightforward that all agents are attracted towards a point inside the convex hull of initial positions $\operatorname{conv}(X^1_0,\dots,X^N_0)$ for all times. The dynamics bear resemblance to opinion formation models, which are studied, e.g., in \cite{albi2015kinetic,krause2000discrete,naldi2010mathematical,motschtadmor}  (see also references therein). Further, there is a relation to consensus mechanism in the velocity components of Cucker-Smale type models (see e.g.~\cite{CarrilloCS,cucker1,cucker2,HaTadmor}). 

Rewriting \eqref{eq:particle} into the deterministic ODE system
\begin{align}
 \frac{d X^i}{d t} = \left(\frac{1}{\sum\nolimits_j \omega_f^\alpha(X_t^j)}\sum\nolimits_{j\ne i } (X_t^i-X_t^j)\,\omega_f^\alpha(X_t^j)\right)H^\e(f(X_t^i)-f(v_f)),
\end{align}
gives a representation of the dynamics as a pairwise weighted interaction, which is scaled with the nonlinear and nonlocal Heaviside term. Note, that the interactions are not symmetric and the uniform connectivity of agents can vanish with $\e\rightarrow0$.   

In all settings, one is interested in the formation of consensus patterns.

\begin{definition}
 A stationary state ${\bf X}=(X^1,\ldots,X^N)$ of \eqref{eq:particle} with $X^1=\cdots=X^N$ is called \emph{uniform consensus}. If particles aggregate to a stationary state consisting of several spatially separated concentration points, this is termed \emph{nonuniform consensus} (cf.~\cite{chatterjee1977towards,degroot1974reaching}). 
\end{definition}

In the global minimization problem however, we not only expect our algorithm to converge to a uniform consensus state, but the concentration point should be located \emph{at or near the global minima}. This is a significant difference to the aforementioned opinion or flocking models, where the location of the consensus state is usually not of primary concern. The deterministic approach to algorithm \eqref{eq:particle} fails at this additional requirement, due to the existence of, possibly uncountable, nonuniform consensus stationary states.

\begin{lemma}\label{lem:stationary} 
 Consider \eqref{eq:particle} with $\sigma=0$. Then any selection of positions ${\bf X}=(X^1,\ldots,X^N)$ taken from any level set of $f$ such that $f(X^1)=\dots=f(X^N)$ is a stationary state.  
\end{lemma}

The inclusion of noise in \eqref{eq:particle} eliminates such unstable equilibria, since the formation of particle configurations that are nonuniform consensus has probability zero due to the Brownian motion. This also means that only uniform consensus is permitted in the stochastic model. In Fig.~\ref{fig:comparisonDetSto} we compare the deterministic case ($\sigma=0$) with a stochastic one ($\sigma=0.7$) for a double-well type function
\begin{align}\label{eq:double-well}
 f(x) = 0.2 x^4 - 2x^2 + 0.5x +10,
\end{align}
with global minimum positioned at $x_* = -2.29613$. One clearly observes that the particles form a nonuniform consensus on the level set $f(X^i) = 9.896$ in the deterministic case. In the presence of noise, the particles converge to a uniform consensus state near the global minimum.

\begin{figure}
\centering
\begin{subfigure}{.49\textwidth}
 \includegraphics[keepaspectratio=true,width=\textwidth]{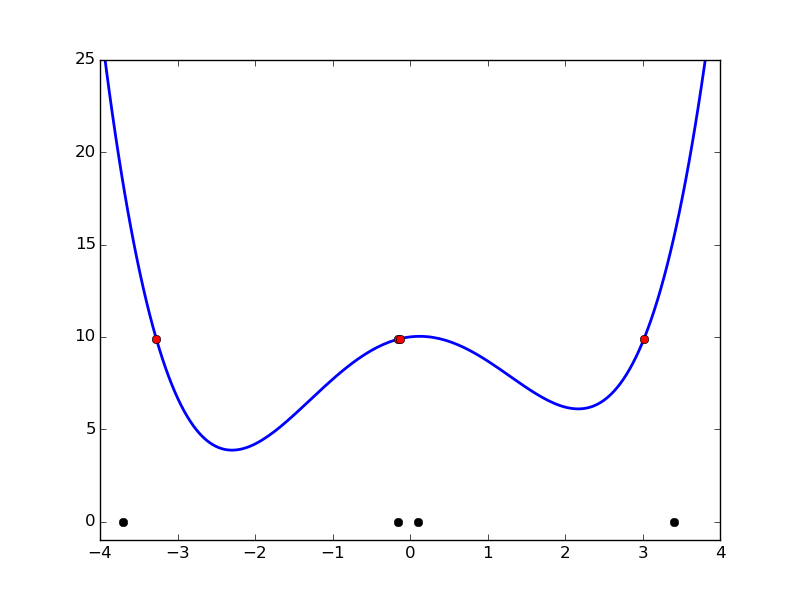}
\end{subfigure}
\begin{subfigure}{.49\textwidth}
 \includegraphics[keepaspectratio=true,width=\textwidth]{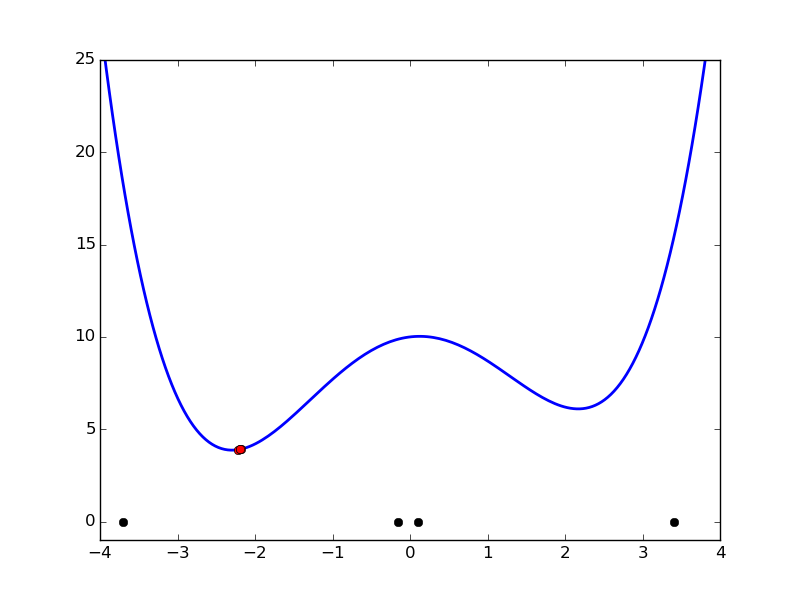}
\end{subfigure}
\caption{Double-well objective function given in \eqref{eq:double-well}. 
The black markers show the initial particle positions while the red markers denote the converged particles positions mapped onto the objective function. Left: deterministic scheme ($\sigma=0$). Right: stochastic scheme ($\sigma=0.7$).}\label{fig:comparisonDetSto}
\end{figure}

\subsection{The mean-field limit}\label{sec:mean-field}
Due to the smoothness of the right-hand side, we can formally derive the mean-field equation of the microscopic system \eqref{eq:particle}. Mean-field limits for interacting particle system with noise have been rigorously studied in, e.g.~\cite{CarrilloCanizoBolley,dobrushin1989dynamical,Sznitman}.

A standard strategy to formally derive the mean-field limit is to consider the dynamics of the first marginal of $\nu_t^N$ and make the so-called {\em propagation of chaos} assumption on the marginals. More specifically, we assume that $\nu_t^N \approx \rho_t^{\otimes N}$ for $N\gg 1$, i.e., the random variables $X_t^1,\ldots,X_t^N$ are approximately independently $\rho_t$-distributed. In this case,
\[
 \frac{1}{N}\sum\nolimits_j \omega_f(X_t^j) \approx \int_{\rr^d} \omega_f\,d\rho_t,\qquad \frac{1}{N}\sum\nolimits_j X_t^j\,\omega_f(X_t^j) \approx \int_{\rr^d} x\,\omega_f\,d\rho_t,
\]
simply due to the law of large numbers, and therefore $v_f \approx v_f[\rho_t]$. Consequently, \eqref{model-dynamics} becomes independent of $j\ne i$, and we obtain the so-called {\em Mc-Kean nonlinear process}
\begin{subequations}
\begin{align}
 d\X_t = -\lambda (\X_t-v_f[\rho_t])\,H^\e(f(\X_t)-f(v_f[\rho_t]))\,d t + \sqrt{2} \sigma |\X_t-v_f[\rho_t]| d W_t
 \end{align}
 where the weighted average is given by
 \begin{align}
 v_f[\rho_t ] = \frac{1}{\int_{\rr^d} \omega_f^\alpha d\rho_t} \int_{\rr^d} x\,\omega_f^\alpha d\rho_t,\qquad \rho_t = \text{law}(\X_t).
\end{align}
\end{subequations}
Equation (2.9a) may be equivalently expressed as the Fokker--Planck equation
\begin{gather}\label{eq:model-pde}
 \partial_t \rho_t = \Delta (\kappa[\rho_t]\rho_t) + \text{div}(\mu[\rho_t]\rho_t), \\
 \intertext{with}
 \kappa[\rho_t](x) = \sigma^2|x-v_f[\rho_t]|^2,\qquad \mu[\rho_t](x) = -\lambda(x-v_f[\rho_t])H^\e(f(x)-f(v_f[\rho_t])),\nonumber
\end{gather}
which describes the evolution of the law corresponding to the Mc-Kean nonlinear process $\{\X_t\in\rr^d\,|\,t\ge 0\}$. 

\begin{remark}
We note that the presence of $v_f$ makes the Fokker--Planck equation nonlinear and nonlocal in both the convection and diffusion part. This is nonstandard in the literature and raises several analytical and numerical questions \cite{CarrilloChoiTotzeckTse}.
\end{remark}

The mean-field limit equation \eqref{eq:model-pde} allows for an analytical discussion of the dynamics, which is advantageous compared to the large-scale stochastic system \eqref{eq:particle}. The formation of a consensus corresponds to a concentration of $\rho_t$ in form of a Dirac measure as time tends to infinity. While both rigorous existence theory and the concentration phenomena of the diffusion process are difficult to study due the degeneracy at $v_f$, we provide two simple results for the case $\sigma=0$ and in the absence of the Heaviside function, i.e., $H^\e\equiv 1$, which indicate the usability of the method. For analytical results for the case $\sigma>0$ and $H^\e\equiv 1$, we refer to \cite{CarrilloChoiTotzeckTse}.

\begin{lemma}\label{lem:concentration}
 Let $\alpha>0$ be arbitrary but fixed and $\rho^\alpha\in\C([0,\infty),\P_2^{ac}(\rr^d))$ satisfy \eqref{eq:model-pde} with $\sigma=0$ and $H^\e\equiv 1$. Then $\rho_t^\alpha \to \delta_{\hat x}$ in the sense of distributions as $t\to\infty$ for some $\hat x\in\rr^d$.
\end{lemma}
\begin{proof}
A simple computation of the evolution of the variance gives
\begin{align*}
 \frac{d}{dt}V(\rho_t^\alpha) &:=\frac{d}{dt}\iint |x-y|^2 d\rho_t^\alpha(x)d\rho_t^\alpha(y) \\
 &= -2\lambda\iint |x-y|^2\,d\rho_t^\alpha(x)\,d\rho_t^\alpha(y) = -2\lambda V(\rho_t^\alpha),
\end{align*}
which directly shows that $V(\rho_t^\alpha)=V(\rho_0^\alpha)e^{-2\lambda t}\to 0$ as $t\to \infty$. Consequently, for any $\e>0$, we have
\[
 \rho_t^\alpha(\{x\in\rr^d\,|\, |x-\E[\X_t]|\ge \e\}) \le \frac{1}{\e^2}\int_{\rr^d} |x-\E[\X_t]|^2d\rho_t^\alpha = \frac{1}{\e^2}V(\rho_t^\alpha).
\]
Passing to the limit $t\to\infty$ on the right-hand side shows the concentration of the measure $\rho_t^\alpha$ around its expectation. To show that its expectation $\E[\X_t]$ converges towards some $\hat x\in\rr^d$ as $t\to\infty$, we compute the evolution of $\E[\X_t]$, which yields
\[
 \frac{d}{dt} \E[\X_t] = -\lambda \int_{\rr^d} (x - v_f[\rho_t^\alpha])\,d\rho_t^\alpha = -\lambda (\E[\X_t] - v_f[\rho_t^\alpha]).
\]
On the other hand, we have
\[
 \left| \E[\X_t] - v_f[\rho_t^\alpha] \right|^2 \le \int_{\rr^d} |x-v_f[\rho_t^\alpha]|^2 d\rho_t^\alpha \le  c_f V(\rho_t^\alpha) \le c_f V(\rho_0^\alpha)e^{-2\lambda t},
\]
with $c_f = \exp(\alpha(\sup f - \inf f))$. Therefore, we further obtain
\begin{align*}
 |\E[\X_t] - \E[\X_0]| &\le \lambda\int_0^t \int_{\rr^d} |x - v_f[\rho_s^\alpha]|\,d\rho_s^\alpha \,ds \le \lambda c_f V(\rho_0^\alpha) \int_0^t e^{-2\lambda s}\,ds \\
 &= (c_f/2) V(\rho_0^\alpha)\left(1-e^{-2\lambda t}\right),
\end{align*}
which shows that $|\hat x - \E[\X_0]| = \lim_{t\to\infty}|\E[\X_t] - \E[\X_0]| \le (c_f/2) V(\rho_0^\alpha)$, i.e., $\hat x\in \rr^d$. Furthermore, since $|\E[\X_t] - v_f[\rho_t^\alpha]|\to 0$, we also have that $\lim_{t\to\infty} v_f[\rho_t^\alpha] = \hat x$.
\end{proof}

By assuming more regularity on the objective function $f$ we can even show that the point $v_f$ converges, as $t\to\infty$, towards a neighborhood of the global minimizer $x_*$. Furthermore, this neighborhood may be made arbitrarily small by choosing $\alpha\gg 1$ sufficiently large.

\begin{lemma}
Let $f$ have the form $f=g + \chi$, where
\begin{enumerate}
 \item $g\in\C^2(\rr^d)$ globally strongly convex, i.e.,
 \[
  \langle x-y,\nabla g(x)-\nabla g(y)\rangle \ge m_g|x-y|^2,
 \]
 for some constant $m_g>0$.
 
 \item $\nabla\chi$ is globally Lipschitz continuous and bounded with $c_\chi := \text{Lip}(\nabla \chi) \le m_g$, where $\text{Lip}(\nabla \chi)$ denotes the Lipschitz constant of $\nabla\chi$.
\end{enumerate}
 Then, for any $\e>0$, there exists an $\bar \alpha>0$, such that 
 \[
  \lim\nolimits_{t\to\infty} f(v_f[\rho_t^{\bar \alpha}]) \le f(x_*) + \e,
 \]
 where $\rho^{\bar \alpha}\in\C([0,\infty),\P_2^{ac}(\rr^d))$ is a solution of \eqref{eq:model-pde} with $\sigma=0$ and $H^\e\equiv 1$. In particular, there exists some $\delta>0$ such that $\lim\nolimits_{t\to\infty} v_f[\rho_t^{\bar \alpha}]\in B_\delta(x_*)$.
\end{lemma}
\begin{proof}
 We begin the proof by estimating the evolution of the following functional
 \[
  \F_\alpha(\rho_t^\alpha):= \log\left(\int_{\rr^d} e^{-\alpha f} d\rho_t^\alpha \right).
 \]
 Taking its time derivative gives
 \begin{align*}
  \frac{d}{dt} \F_\alpha(\rho_t^\alpha) 
  &= \lambda\alpha\int_{\rr^d} \langle \nabla g(x)-\nabla g(v_f[\rho_t^\alpha]),x-v_f[\rho_t^\alpha]\rangle\,d\eta_t^\alpha \\
  &\hspace*{12em}+ \lambda\int_{\rr^d} \nabla \chi(x)\cdot (x-v_f[\rho_t^\alpha])\,d\eta_t^\alpha,
 \end{align*}
 where we used the fact that
 \[
  \int_{\rr^d} (x-v_f[\rho_t^\alpha])\,d\eta_t^\alpha = 0.
 \]
 Using the assumptions on $f$, we estimate from below to obtain 
 \[
  \frac{d}{dt} \F_\alpha(\rho_t^\alpha) \ge \lambda\alpha\big(m_g -c_\chi\big)\int_{\rr^d} |x-v_f[\rho_t^\alpha]|^2 d\eta_t^\alpha \ge 0,
 \]
 which directly implies that
 \begin{align}\label{eq:uniform}
  -\log\left(\int_{\rr^d} e^{-\alpha f} d\rho_t^\alpha \right) = -\F_\alpha(\rho_t^\alpha) \le -\F_\alpha(\rho_0) = -\log\left(\int_{\rr^d} e^{-\alpha f} d\rho_0 \right),
 \end{align}
 for all $t\ge 0$. On the other hand, we have by definition that
 \[
  \frac{d}{d\alpha}\F_\alpha(\rho_t^\alpha) = -\int_{\rr^d} f\,d\eta_t^\alpha,
 \]
 and by differentiation of \eqref{eq:uniform} w.r.t.~$\alpha$, we obtain
 \[
  \int_{\rr^d} f\,d\eta_t^\alpha \le \int_{\rr^d} f\,d\eta_0.
 \]
 From Proposition~\ref{prop:laplace}, we further obtain the existence of an $\bar \alpha\gg 1$ such that 
 \[
  \int_{\rr^d} (f-f_*)\,d\eta_t^{\bar \alpha} \le \int_{\rr^d} (f-f_*)\,d\eta_0 \le \e.
 \]
 We now proceed to estimate $f(v_f[\rho_t^{\bar\alpha}])$ as follows:
 \begin{align*}
  f(v_f[\rho_t^{\bar\alpha}]) &= g(v_f[\rho_t^{\bar\alpha}]) + \chi(v_f[\rho_t^{\bar\alpha}]) \le \int_{\rr^d} g(x)\,d\eta_t^{\bar\alpha} + \chi(v_f[\rho_t^{\bar\alpha}]) \\
  &= \int_{\rr^d} f(x)\,d\eta_t^{\bar\alpha} + \chi(v_f[\rho_t^{\bar\alpha}]) - \int_{\rr^d} \chi(x)\,d\eta_t^{\bar\alpha}  \\
  &\le \int_{\rr^d} f(x)\,d\eta_t^{\bar\alpha} + \|\nabla\chi\|_\infty\int_{\rr^d} |x-v_f[\rho_t^{\bar\alpha}]|d\eta_t^{\bar\alpha},
 \end{align*}
 where we made use of Jensen's inequality in the first inequality. Consequently,
 \begin{align*}
  f(v_f[\rho_t^{\bar\alpha}]) - f_* &\le \int_{\rr^d} (f-f_*)\,d\eta_t^{\bar\alpha} + \|\nabla\chi\|_\infty\int_{\rr^d} |x-v_f[\rho_t^{\bar\alpha}]|d\eta_t^{\bar\alpha} \\
  &\le \e + \|\nabla\chi\|_\infty\int_{\rr^d} |x-v_f[\rho_t^{\bar\alpha}]|d\eta_t^{\bar\alpha}.
 \end{align*}
 Since the last term on the right hand side converges to zero as $t\to\infty$ (cf.~Lemma~\ref{lem:concentration}), we may pass to the limit to obtain
 \[
  \lim\nolimits_{t\to\infty} f(v_f[\rho_t^{\bar\alpha}]) \le f_* + \e.
 \]
 Due to continuity of $f$, we find some $\delta>0$ such that $\lim_{t\to\infty} v_f[\rho_t^{\bar\alpha}] \in B_\delta(x_*)$, where $\bar\alpha$ needs to be chosen even larger if necessary.
\end{proof}

\section{Numerical Experiments in $d=1$}\label{sec:numerics}
In this section we study the performance of the consensus-based optimization algorithm and investigate, in particular, the relation of the particle system and the mean-field PDE. We employ two standard test cases from optimization literature \cite{jamil2013literature}, namely the \emph{Ackley function}  
\begin{equation}\label{eq-ackley}
 f_A(x) = -20 \exp\!\bigg(\!\!-\frac{0.2}{\sqrt{d}}\|x-B\|\bigg) - \exp\!\bigg(\frac{1}{d}\sum_{i=1}^d \cos(2\pi (x_i-B) )\bigg) + 20 + e + C
\end{equation}
and the \emph{Rastrigin function} 
\begin{equation}\label{eq-rastrigin}
 f_R(x)=  \frac{1}{d} \sum_{i=1}^d \left[(x_i-B)^2 - 10 \cos(2\pi (x_i-B))  + 10 \right] + C, 
\end{equation}
where $d \in \nn$ denotes the dimension of the search space and $B,C\in\rr$ are constant shifts. As seen in Fig.~\ref{fig-3benchmarks}, both functions attain multiple local minima but only one global minimum.
\begin{figure}[h]
\centering
\begin{subfigure}{.49\textwidth}
 \includegraphics[keepaspectratio=true,width=\textwidth]{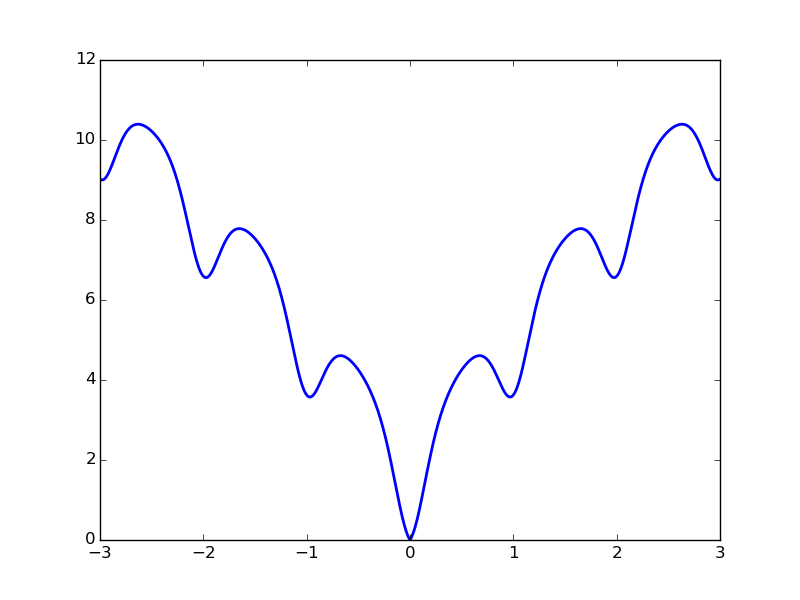}
 \caption{Ackley}
\end{subfigure}
\begin{subfigure}{.49\textwidth}
 \includegraphics[keepaspectratio=true,width=\textwidth]{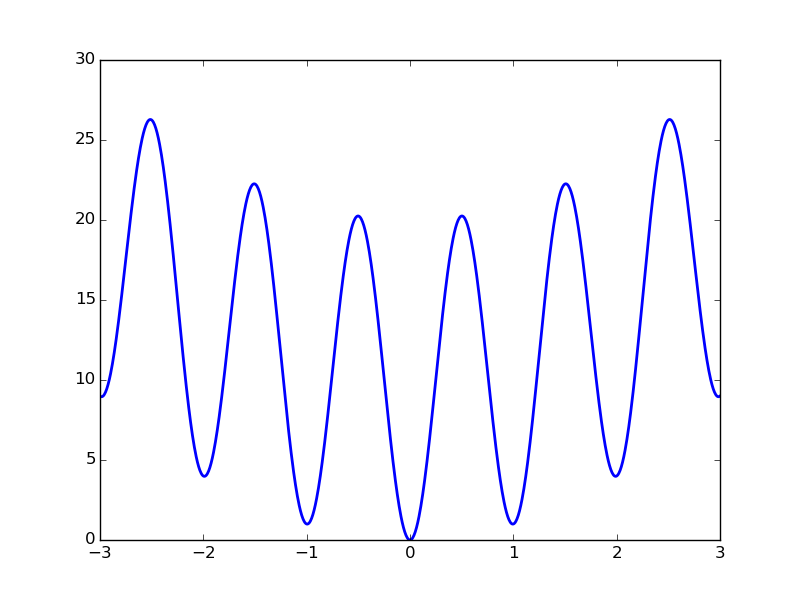}
 \caption{Rastrigin}
\end{subfigure}
\caption{Benchmark functions}\label{fig-3benchmarks}
\end{figure}

\subsection{Numerical methods}
For the stochastic system we use a simple particle scheme, while the mean-field PDE is discretized with the help of a splitting scheme in time and a discontinuous Galerkin method in space. For the numerical experiments we choose $\lambda=1$.
 
\subsubsection{Particle scheme}
The particle simulations are realized using the standard Euler-Maruyama scheme \cite{Higham}. We recall the particle equations:
\begin{equation}\label{eq:EM}
 dX_t^{i} = -\lambda(X_t^i-v_f) H^\epsilon(f(X_t^i)-f(v_f))dt + \sqrt{2}\sigma|X_t^{i}-v_f|dW_t^i,
\end{equation}
where the Heaviside function is approximated by 
\begin{equation*}
 H^\epsilon(x) = \frac{1}{2} \erf(\frac{1}{\epsilon} x) + \frac{1}{2}. 
\end{equation*}
Note the stochasticity of the  particle scheme induced by the Brownian motion $W_t$ in \eqref{eq:EM} (cf. the discussion in Lemma~\ref{lem:stationary}). Thus, to get meaningful results we simulate $M$ samples with different initial data and different realizations of the Brownian motion. As random number generator we use the \texttt{numpy\@.random} package of \texttt{python} which is based on the Mersenne Twister pseudo-random number generator \cite{MersenneTwister}.

\subsubsection{Galerkin scheme for the mean-field equation}
The solver for the mean-field PDE uses a discontinuous Galerkin method in space. A splitting of the transport and diffusion part is used to compute the solution iteratively. In fact, instead of solving the full equation
\begin{equation*}
 \partial_t \rho - \div(\mu[\rho] \rho) = \Delta(\kappa[\rho] \rho ),
\end{equation*}
one half-step in time is computed for the convection part 
\begin{equation}\label{convection}
 \partial_t\rho_* - \div(\mu[\rho_{k-1}] \rho_*) =0,\qquad \rho_*(0)=\rho_{k-1}
\end{equation}
with local Lax--Friedrichs method, followed by one semi-implicit time step for the diffusion part
\begin{equation}\label{diffusion}
 \partial_t\rho_{**} = \Delta (\kappa[\rho_*] \rho_{**}),\qquad \rho_{**}(0)=\rho_{*},
\end{equation}
and finally one half-step of \eqref{convection} again to obtain $\rho_k$. Here $\rho_*$, $\rho_{**}$ denotes the intermediate solutions after the first and second step, respectively. This provides a semi-implicit scheme of second order, typically known as {\em Strang splitting} \cite{Strang}.

Since we expect very large spatial gradients in the solution due to the convergence to a Dirac function (see Lemma \ref{lem:concentration}), we use discontinuous Galerkin elements. Hence, there is need for an appropriate handling of the numerical fluxes. For any smooth test function $\test$ the convective part in weak form reads as
\begin{equation*}
 \int_K \rho_* \,\test\,dx = \int_K \rho_{k-1} \,\test\,dx + \tau\int_K \mu_{k-1} \rho_{k-1} \cdot \nabla \test\,dx - \tau\int_{\partial K} \overline{\mu_{k-1} \rho_{k-1}} \cdot \textbf{n}\, \test\, ds,
\end{equation*}
for any interval $K\subset\rr$ in a given partition $\mathcal T_h$ with $\mu_{k-1}= \mu[\rho_{k-1}]$. As mentioned, the local Lax--Friedrichs numerical flux \cite{cockburn2003discontinuous} is implemented, i.e.,
\begin{equation*}
 \overline{\mu\rho}  = \mu \{\rho\} + \frac{1}{2} |\mu| \ldb \rho \rdb,
\end{equation*}
where 
\begin{equation*}
\{\rho\} = \frac{1}{2}(\rho^+ + \rho^-)\qquad \text{and}\qquad \ldb \rho \rdb = \rho^+\textbf{n} + \rho^- \textbf{n},
\end{equation*}
 are the average and jump operators respectively. For the diffusion part the weak formulation reads
\begin{align*}
 \int_K \rho_{**}\, \test\,dx + \tau \int_K \nabla (\kappa_* \rho_{**}) \cdot \nabla \test\,dx - \tau\int_{\partial K} \overline{\kappa_* \rho_{**}\, \test}\,dx = \int_K \rho_* \test\,dx 
\end{align*}
where $\kappa_*=\kappa[\rho_*]$ and the flux function 
\begin{equation*}
 \overline{\kappa \rho\, \test} = \ldb \test\rdb \{ \nabla (\kappa\rho) \}\cdot \textbf{n} + \ldb \kappa\rho \rdb \{ \nabla \test \} \cdot \textbf{n}  - \ldb \kappa\rho \rdb \ldb \test \rdb,
\end{equation*}
proposed in \cite{kulkarni2007discontinuous} is used.

\subsection{CBO in $d=1$}
In the following we show the results of the computations realized with the numerical methods described above. The parameters for the particle simulation are 
\[
 N=50,\quad dt=10^{-1},\quad \alpha = 40,\quad \sigma = 0.7,\quad M = 500,\quad T = 80.
\]
The initial positions of the particles are chosen randomly due to the uniform distribution in $[-3,3]$ with help of the random number generator \texttt{numpy.random} of the \texttt{python} software. The particle simulation is stopped when the final time $T=80$ is reached, i.e., after $T/dt$ iterations. Due to the stochasticity of the Euler-Maruyama scheme, we compute $M$ samples of each testcase and average over the samples to obtain meaningful results. The parameters for the mean-field simulations are
\[
 \alpha = 40,\quad \sigma = 0.7,\quad \tol = 10^{-3},\quad h = 10^{-2}.
\]
Here, $h$ denotes uniform spatial grid size. The stopping criterion for the mean-field simulation is
\[
 \|\rho_k - \rho_{k-1}\|_{L^2(\Omega)} < \tau_k\,\tol,
\]
where $\rho_k$ denotes the solution at timestep $k$ and $\tau_k$ the current timestep, which is required to satisfy the CFL condition 
\[
 \tau_k < h / \max(|\mu|),
\]
in order to resolve the drift term. The initial distribution of particles and $\rho$ is the uniform distribution on $[-3,3]$. Note that the PDE is deterministic, we therefore need not compute different samples in the mean-field case. We ran the PDE simulation until the stopping criterion was fulfilled in order to fix an appropriate stopping time for the particle simulation. The mean-field simulations for the Ackley benchmark stopped at $T= 5.7$ and $T=25$ in the standard and the shifted case, respectively. The Rastrigin benchmark stopped at $T=58$ and $T=78$ in the corresponding cases. Thus, we fixed the time horizon of the particle simulation at $T=80$.
\begin{figure}[h]
\centering
\begin{subfigure}{.49\textwidth}
 \includegraphics[keepaspectratio=true,width=\textwidth]{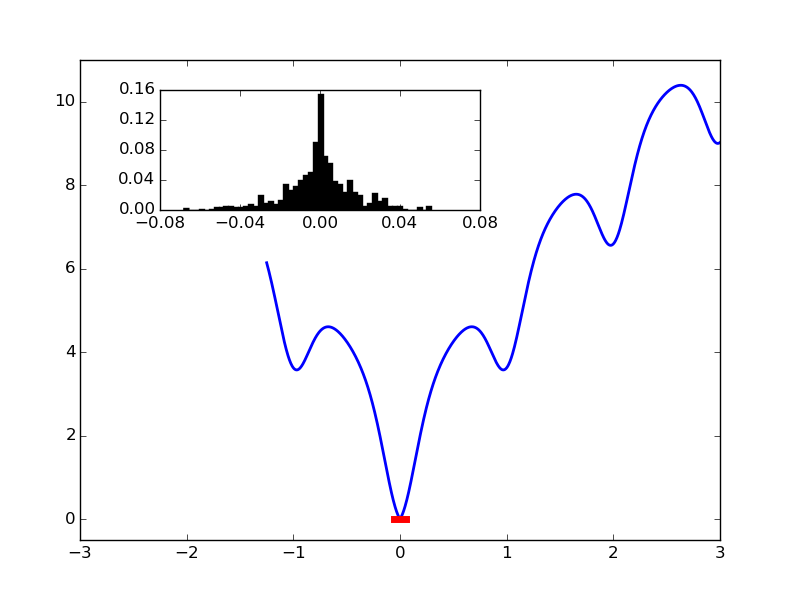}
 \caption{Ackley: $B=C=0$, $\arg\min f_A=0$}
\end{subfigure}
\begin{subfigure}{.49\textwidth}
 \includegraphics[keepaspectratio=true,width=\textwidth]{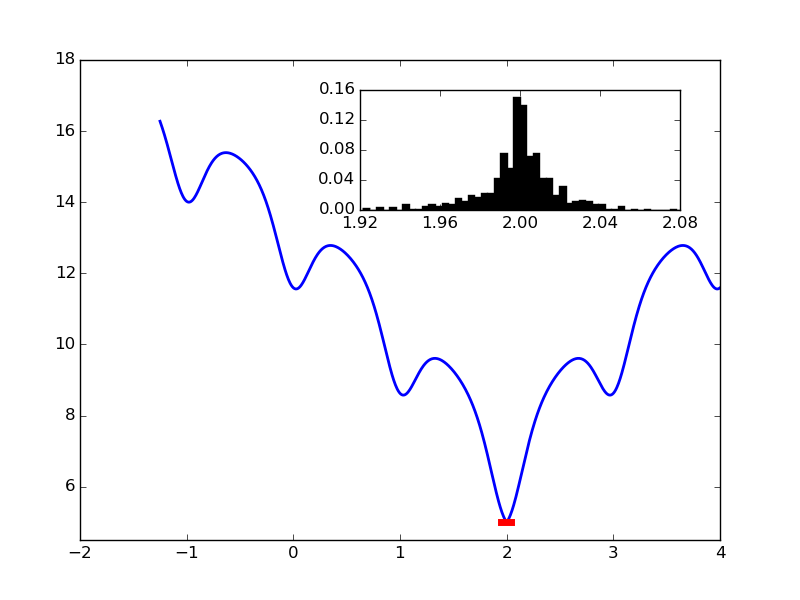}
 \caption{Ackley: $B=2$, $C=5$ $\arg\min f_A=2$}
\end{subfigure}\\
\begin{subfigure}{.49\textwidth}
 \includegraphics[keepaspectratio=true,width=\textwidth]{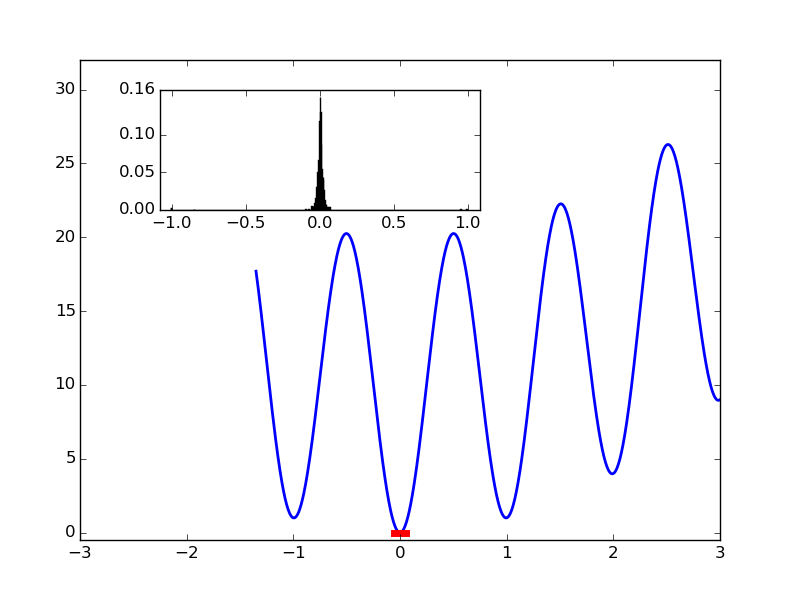}
 \caption{Rastrigin: $B=C=0$, $\arg\min f_R =0$}
\end{subfigure}
\begin{subfigure}{.49\textwidth}
 \includegraphics[keepaspectratio=true,width=\textwidth]{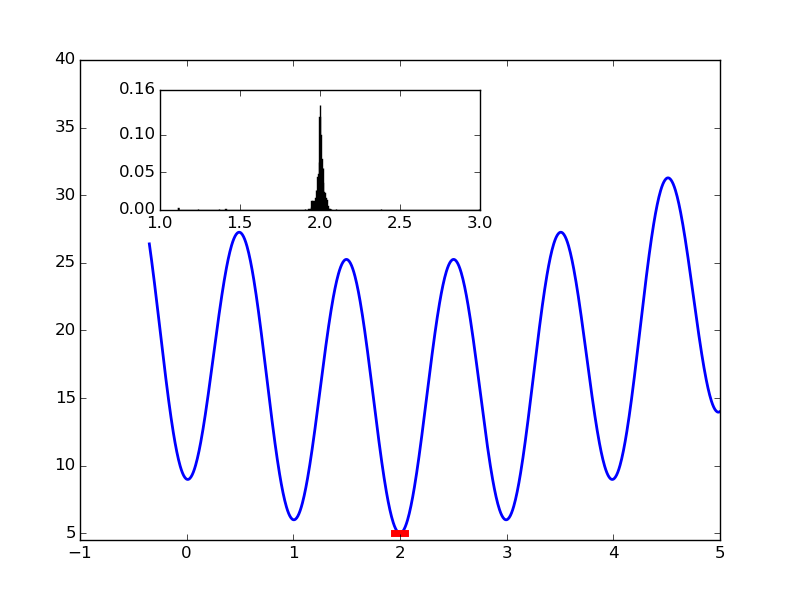}
 \caption{Rastrigin: $B=2$, $C=5$, $\arg\min f_R = 2$}
\end{subfigure}
\caption{Optimization results in 1d. The support of the mean-field density is depicted in red. The histograms show the distribution of $v_f$ for the 500 samples. The histogram for (a) and (b) contains 50 bins. The histogram of (c) and (d) has 500 bins, since $v_f$ is located at a local minima neighboring the global minimum at time $T=80$ in a few samples.}
\label{fig:1Dbenchmarks}
\end{figure}

Figure~\ref{fig:1Dbenchmarks} shows the results for the particle and the mean-field simulations in spatial dimension one. The objective function is depicted in blue, the red line denotes the support of the PDE solution at final time. The histogram in the upper left corner of the plot shows the final positions of the $v_f$ computed by $M=500$ runs with different random initial data.

\begin{remark}
Note, that due to the diffusion, the support of $\rho$ will not concentrate into a single point in the given time interval $[0,T]$. Still we expect the density to concentrate in a small region. Hence, in Fig.~\ref{fig:1Dbenchmarks} we denote the domain where $\rho > 10^{-6}$ by support. 
\end{remark}

The shifted case is interesting since the location of the minimum is less centred with respect to the initial distribution. The proposed algorithm determines a remarkable approximation of the minimum of the Ackley function in the unshifted as well as in the shifted case in every sample of the particle problem. Note that the minimizer candidates $v_f$ are located close to the center of the support resulting from the mean-field simulation for all samples. This last property holds for the Rastrigin function as well. Here, the range of $v_f$ realized by the 500 samples is larger in the shifted case. The realizations of $v_f$ resulting from the simulations are still contained in the support of the mean-field solution except for a few that are located at local minima neighboring the global one. Thus averaging over various realizations allows for a very reasonable approximation of the global minimum.

\subsection{Convergence towards mean-field equation in $d=1$}
In this section, we consider the Ackley function with $C=1$, $x_*=0$, $f_A(x_*)=1$ and particles initially equidistantly distributed on $[-3,1]$. To study the convergence of an evolution of the particle dynamics towards the uniform consensus at $x_*$, we compute the Wasserstein distance \cite{Villani}
\begin{equation}
 W_1\left(\rho_t^N, \delta_{x_*}\right)\quad\text{with}\quad \rho_t^N = \frac{1}{N}\sum\nolimits_i\delta_{X_t^i},
\end{equation}
which is a time-dependent and stochastic quantity, for a sample of $M$ system trajectories. 
Note that in the case $d=1$, one has the representation \cite{Villani}
\begin{equation*}
W_1\left(\rho_t^N, \delta_{x_*}\right) = \int_0^1 |F^{-1}(t) - G^{-1}(t)| dt,
\end{equation*}
where $F$ and $G$ are the cumulative distribution functions associated with $\rho_t^N$ and $\delta_{x_*}$, respectively. This formulation is used to compute the results in the following. For the limiting continuous case with $\rho_0=\mathcal{U}([-3,1])$, we have $W_1(\rho_0,\delta_0)=1.25$. 
\begin{figure}[h]
\centering
\begin{subfigure}{0.49\textwidth}
 \includegraphics[keepaspectratio=true,width=\textwidth]{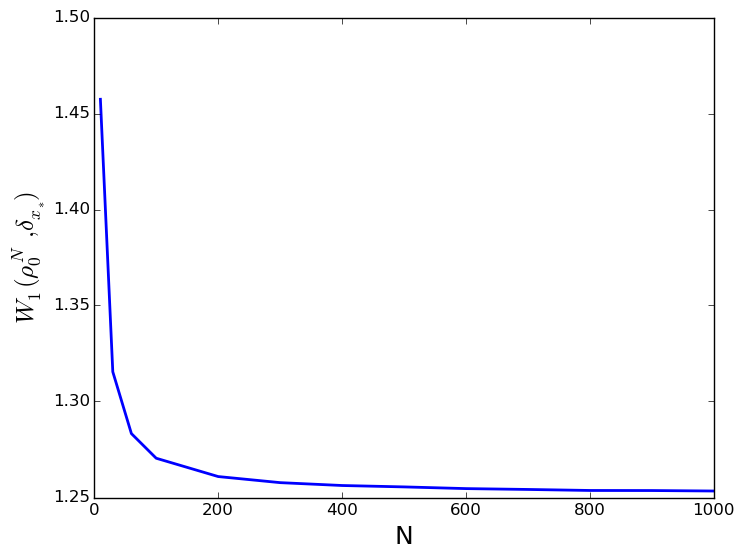}
 \caption{$W_1(\rho_0^N, \delta_{x_*})$}
\end{subfigure}
\begin{subfigure}{0.49\textwidth}
 \includegraphics[keepaspectratio=true,width=\textwidth]{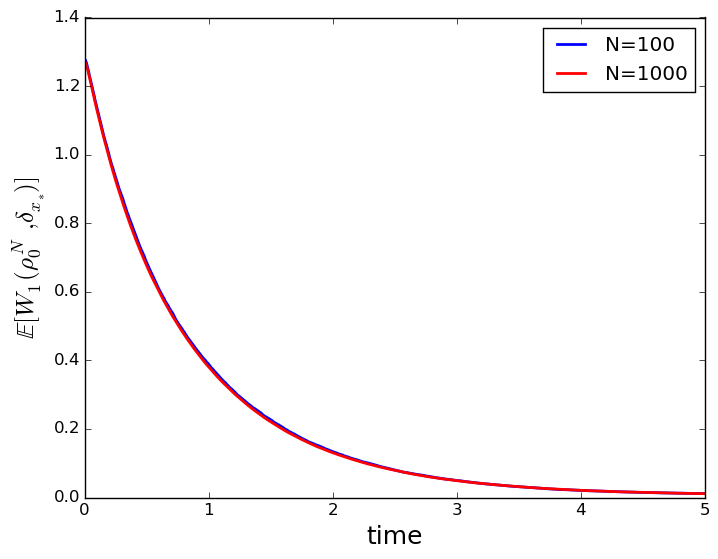}
 \caption{$\E[W_1(\rho_t^N,\delta_{x_*})]$}
\end{subfigure}\\
\begin{subfigure}{0.49\textwidth}
 \includegraphics[keepaspectratio=true,width=\textwidth]{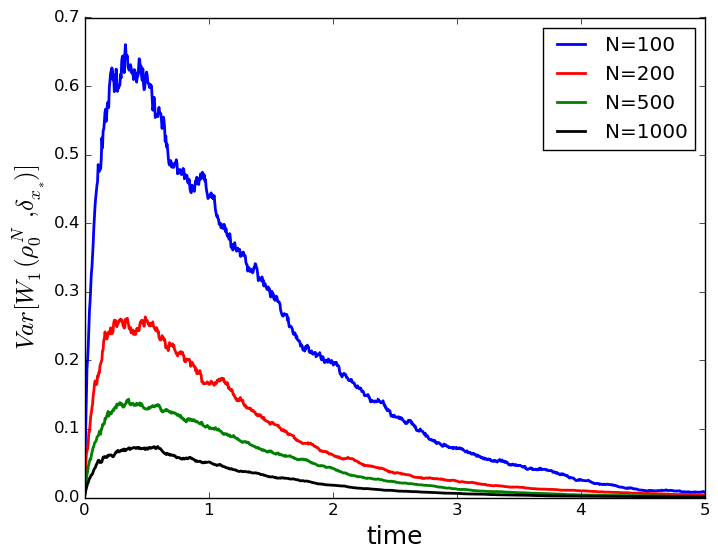}
 \caption{$\text{Var}[W_1(\rho_t^N, \delta_{x_*})]$}
\end{subfigure}
\caption{Convergence of CBO \eqref{eq:particle} to uniform consensus for the 1d Ackley function.}\label{fig-w1}
\end{figure}

Fig.~\ref{fig-w1}(a) shows that the $W_1$-distance of the particle configuration at $t=0$ to the uniform consensus $\delta_{x_*}$ approaches the value in the continuous uniformly distributed case as the number of particles increases. Fig.~\ref{fig-w1}(b) depicts an exponential decay of the mean of the $W_1$ distance for varying number of particles $N=100,1000$. The mean is taken over $M=1000$ realizations. Notice that the expectations are essentially identical for the different values of $N$. The variance of the $W_1$-distance is initially zero and decays after a peak, as seen in Fig.~\ref{fig-w1}(c). Furthermore, it decreases uniformly as $N\rightarrow\infty$, highlighting that the dynamics are governed by the deterministic mean-field equation in the limit.

\section{Numerical Experiments in Multiple Dimensions ($d=20$)}\label{sec:multi}
Efficient algorithms for global optimization become practically relevant for high-dimensional problems, where grid-based procedures exceed computational capabilities. We illustrate the performance of the CBO algorithm \eqref{eq:particle} with several toy experiments in dimension $d=20$. 

Assuming that the solution to the mean-field equation shows concentration arbitrarily close to $x_*$ for sufficiently large $\alpha$, the particle dynamics will do so too. However, we theoretically expect the number of particles necessary to mimic the mean-field dynamics to scale with $\mathcal{O}(N^{d/2})$. Strikingly, good optimization results are already obtained for a \emph{much smaller} number of particles, as shown in the following examples using the Ackley function and the Rastrigin function as benchmarks.

In order to classify the numerical results we will make use of a success rate and the expectation of $\frac{1}{d}\norm{v_f - x_*}^2$ computed with $M =1000$ realizations of the CBO algorithm. Due to the radial symmetry of the benchmark function the evaluation of the norm is justified. We consider a run to be successful if $v_f\in Q_{0.25}(x_*)=(x_*-0.25,x_*+0.25)^d$, i.e., our candidate for the minimizer is contained in the open $\|\cdot\|_\infty$-ball with radius 0.25 around the true minimizer $x_*$. Note that this condition is more restrictive than requiring $v_f$ to be located closer to the minimum than the closest local maximum, which would be $v_f \in Q_{0.5}(x_*)$.  This would be a valid definition for success as well, since a steepest descent search initialized in $v_f$ would find the global minimizer. Nevertheless, by this definition all simulation results would lead to successful runs and we would not get an impression of the influence of the parameters. This is the reason we use the more restrictive definition of success. The following parameters are used for the simulations
\begin{center}
$d=20$,\, $T=10$,\, $M=1000$,\, $dt=0.01$,\, and $\sigma = 5$.
\end{center}
To initialize the particle positions we draw random variables that are uniformly distributed in the hypercube $[-3,3]^d$. Thus, for $x_*=0$ the minimizer is located at the center of the initial data. 

\subsection{Variations in particle number $N\in\nn$}

We begin by investigating the influence of the number of particles $N$ on the numerical results. Table~\ref{tab-D20-A} and Table~\ref{tab-D20-R} show the success rate of the runs and the expectation of the distances $\frac{1}{d}\norm{v_f - x_*}^2$ for the Ackley and the Rastrigin benchmark functions, respectively. In case of the Ackley function the success rate is $100\%$ for all simulations. The values of the expectations show a slight difference for varying parameters. The approximation is better if the minimizer $x_*$ is centered with respect to the initial data, i.e.~$x_*=0$. 
\begin{table}[htb]\centering
\caption{Ackley function in $d=20$ with $\alpha=30$.}
\begin{tabular}{c | c || c c c }
& & & $N$ & \\
 $x_*$ & & 50 & 100 & 200 \\
 \hline
0 & \text{success rate}   & 100\% & 100\%& 100\% \\ 
& $\frac{1}{d}\mathbb{E}[\norm{v_f(T)-x_*}^2]$ & $5.21e^{-4}$ & $1.18e^{-3}$ & $2.47e^{-3}$ \\ 
\hline
1 & \text{success rate}   & 100\%& 100\%& 100\% \\ 
& $\frac{1}{d}\mathbb{E}[\norm{v_f(T)-x_*}^2]$& $5.23e^{-4}$ & $1.21e^{-3}$ & $2.55e^{-3}$ \\ 
\hline
 2 & \text{success rate} & 100\% & 100\% & 100\% \\ 
& $\frac{1}{d}\mathbb{E}[\norm{v_f(T)-x_*}^2]$ & $5.46e^{-4}$ & $1.24e^{-3}$ & $2.57e^{-3}$
\end{tabular}
\label{tab-D20-A}
\end{table}

The results are very different for the Rastrigin benchmark. The success rate is between $34\%$ and $63\%$. It is better for larger $N$ and for $x_*$ centered w.r.t.~the initial data. This difference in the behavior is anticipated by the theory, since the objective of the weighted average  is to separate local minima from the global minima. This separation works better if the distance between the local minima and the global one is large. In case of the Ackley function this distance is larger compared to the one of the Rastrigin function, which explains the different results. In Fig.~\ref{fig:evolutionDiffN} we see the evolution of the expectation $\frac{1}{d}\mathbb{E}[\norm{v_f - x_*}^2]$. For both benchmarks, the simulation with $N=200$ performs best. The smaller $N$, the more time is needed to approximate the minimizer appropriately. 
\begin{table}[htb]\centering
\caption{Rastrigin function in $d=20$ with $\alpha=30$.}
\begin{tabular}{c | c || c c c }
& & & $N$ & \\
 $x_*$ & & 50 & 100 & 200 \\
 \hline
0 & \text{success rate}   & 34.\% & 61.1\%& 62.2\% \\ 
& $\frac{1}{d}\mathbb{E}[\norm{v_f(T)-x_*}^2]$ & $3.12e^{-1}$ & $2.47e^{-1}$ & $2.42e^{-1}$ \\ 
\hline
1 & \text{success rate}   & 34.5\%& 57.1\%& 61.6\% \\ 
& $\frac{1}{d}\mathbb{E}[\norm{v_f(T)-x_*}^2]$& $3.09e^{-1}$ & $2.52e^{-1}$ & $0.244e^{-1}$ \\ 
\hline
 2 & \text{success rate} & 35.5\% & 54.8\% & 62.4\% \\ 
& $\frac{1}{d}\mathbb{E}[\norm{v_f(T)-x_*}^2]$ & $3.06e^{-1}$ & $2.51e^{-1}$ & $2.44e^{-1}$
\end{tabular}
\label{tab-D20-R}
\end{table}

\begin{figure}[h]
\centering
\begin{subfigure}{.49\textwidth}
 \includegraphics[keepaspectratio=true,width=\textwidth]{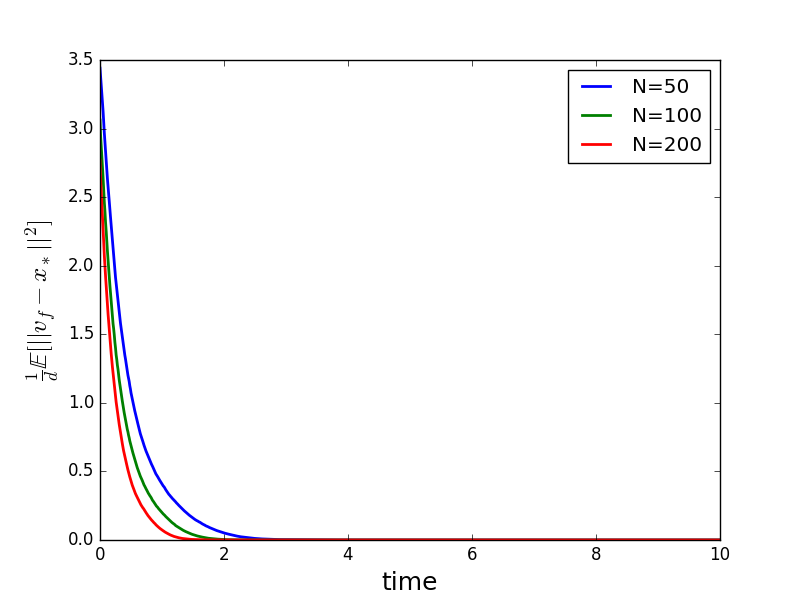}
 \caption{Ackley}
\end{subfigure}
\begin{subfigure}{.49\textwidth}
 \includegraphics[keepaspectratio=true,width=\textwidth]{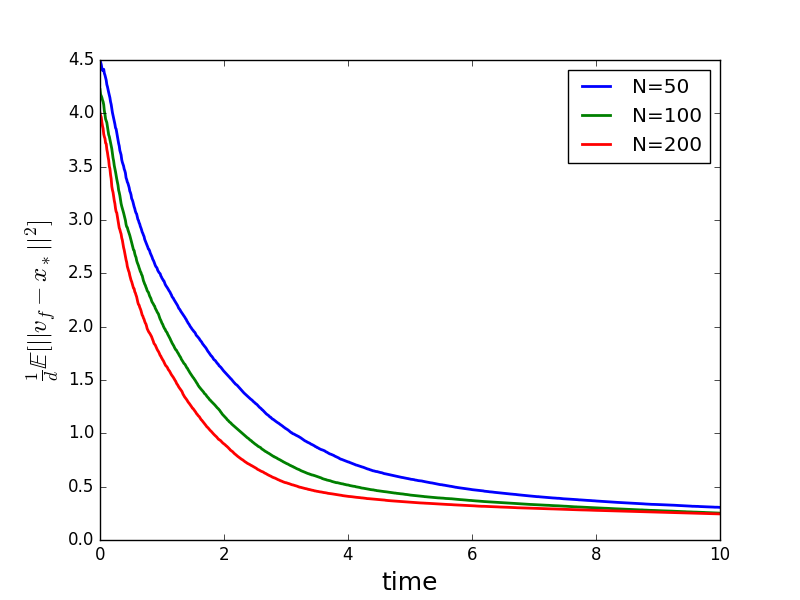}
 \caption{Rastrigin}
\end{subfigure}
\caption{Evolution of the expectation of the distances of $v_f$ and $x_*$ over time.}
\label{fig:evolutionDiffN}
\end{figure}

\subsection{Variations in exponential weight $\alpha$}

From the theory, we expect that the separation of local and global minima works better for larger $\alpha\gg 1$. Thus, we study the influence of the weight parameter in the following and check if the results improve for larger $\alpha>0$. In Table~\ref{tab-D20-alphaA} the results for the Ackley function with $N=100$ are shown. Again, the performance is very good for all tested cases. Moreover, the expectation shows that $v_f$ provides a very good approximation of the global minimizer, even for small weight parameters $\alpha$.
\begin{table}[htb]\centering
\caption{Ackley function in $d=20$ with $N=100$.}
\begin{tabular}{c | c || c c c c c }
& & & & $\alpha$ & & \\
 $x_*$ & & 10 & 20 & 30 & 40 & 50 \\
 \hline
0 & \text{success rate}   & 100\% & 100\%& 100\% & 100\% & 100\% \\ 
& $\frac{1}{d}\mathbb{E}[\norm{v_f(T)-x_*}^2]$ & $2.55e^{-4}$ & $1.06e^{-4}$ & $6.18e^{-5}$ &  $4.21e^{-5}$ & $3.04e^{-5}$ \\ 
\hline
1 & \text{success rate}   & 100\% & 100\%& 100\% & 100\% & 100\% \\ 
&  $\frac{1}{d}\mathbb{E}[\norm{v_f(T)-x_*}^2]$&  $2.58e^{-4}$  & $1.09e^{-4}$ & $6.31e^{-5}$ &  $4.24e^{-5}$  & $3.04e^{-5}$  \\ 
\hline
 2 & \text{success rate} & 100\% & 100\%& 100\% & 100\% & 100\% \\ 
&  $\frac{1}{d}\mathbb{E}[\norm{v_f(T)-x_*}^2]$ & $2.62e^{-4}$ & $1.10e^{-4}$ & $6.46e^{-5}$ & $ 4.35e^{-5}$ & $3.18e^{-5}$ \\ 
\end{tabular}
\label{tab-D20-alphaA}
\end{table}

The results obtained for the Rastrigin function are given in Table~\ref{tab-D20-alphaR}. Here we can see a significant influence of $\alpha>0$. For $\alpha=10$ the performance of the CBO algorithm is rather poor. Nevertheless, for increasing $\alpha$ the results strongly improve. And for $\alpha=50$ the results are already very good. Similarly, the expectation shows that $v_f$ provides an increasingly better approximation of the global minima for larger weight parameters $\alpha\gg 1$. This underlines our theoretical results which indicate that the performance of the CBO algorithm strongly depends on the exponential weight parameter $\alpha>0$.

\begin{table}[htb]
\caption{Rastrigin function in $d=20$ with $N=100$. }
\begin{tabular}{c | c || c c c c c }
& & & & $\alpha$ & & \\
 $x_*$ & & 10 & 20 & 30 & 40 & 50 \\
 \hline
0 & \text{success rate}   & 7.0\%& 11.6\% & 61.1\% & 93.8\% & 99.7\%\\ 
& $\frac{1}{d}\mathbb{E}[\norm{v_f(T)-x_*}^2]$ & $4.32e^{-1}$ & $3.91e^{-1}$ & $2.48e^{-1}$ &  $1.35e^{-1}$ & $7.67e^{-2}$  \\ 
\hline
1 & \text{success rate}   & 7.5\%& 10.3\%& 57.1\% & 96\% & 99.5\% \\
&  $\frac{1}{d}\mathbb{E}[\norm{v_f(T)-x_*}^2]$& $4.34e^{-1}$ & $3.94e^{-1}$ & $2.52e^{-1}$ &  $1.32e^{-1}$ & $7.59e^{-2}$  \\ 
\hline
 2 & \text{success rate} & 7.5\% & 12.2\% & 54.8\% & 94.9\% & 99.3\% \\ 
&  $\frac{1}{d}\mathbb{E}[\norm{v_f(T)-x_*}^2]$ & $4.31e^{-1}$ & $3.92e^{-1}$ & $2.51e^{-1}$ & $1.32^{-1}$ & $7.81e^{-2}$
\end{tabular}
\label{tab-D20-alphaR}
\end{table}

In Fig.~\ref{fig:evolutionNormAlpha} we see the evolution of the expectation $\frac{1}{d}\mathbb{E}[\norm{v_f - x_*}^2]$. As anticipated by the theory and the results shown in the tables, we find an improvement in performance for larger values of $\alpha$. In contrast to the graphs showing this evolution for different $N$, we find that the influence of $\alpha$ becomes more important at the end of the simulation.

\begin{figure}[h]
\centering
\begin{subfigure}{.49\textwidth}
 \includegraphics[keepaspectratio=true,width=\textwidth]{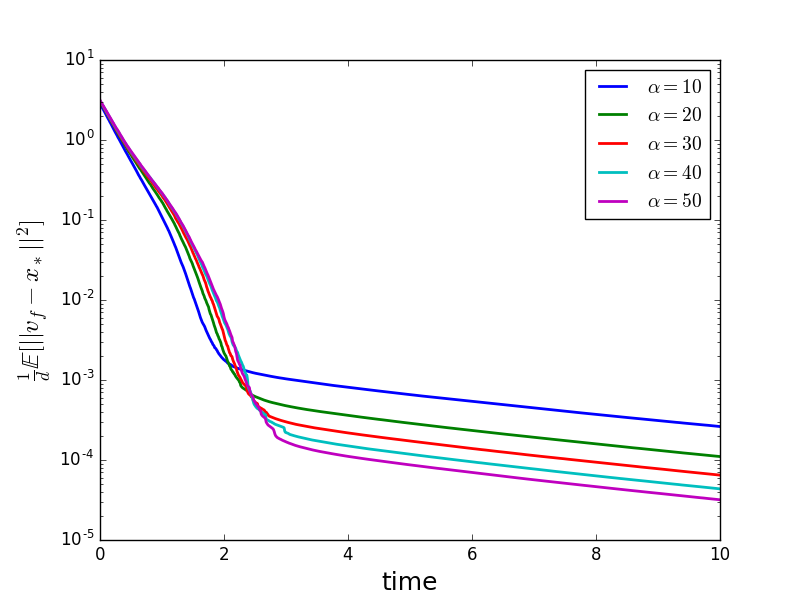}
 \caption{Ackley}
\end{subfigure}
\begin{subfigure}{.49\textwidth}
 \includegraphics[keepaspectratio=true,width=\textwidth]{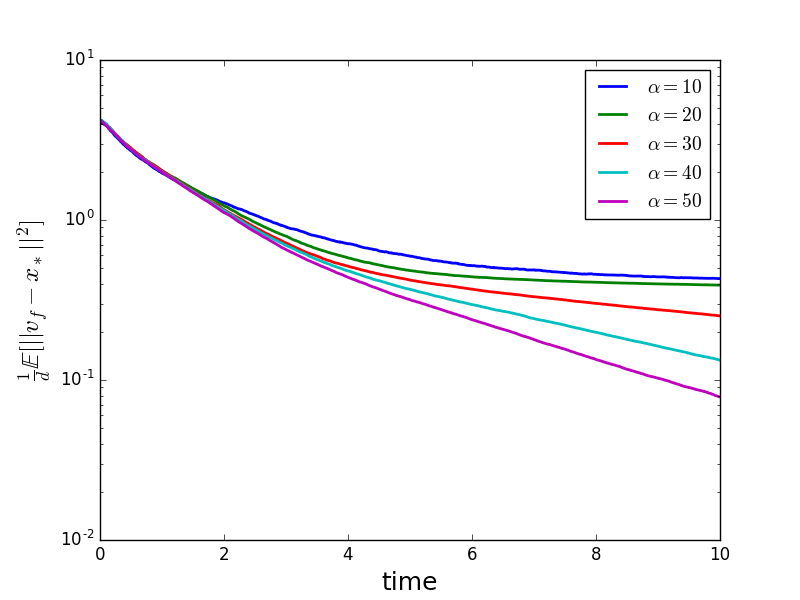}
 \caption{Rastrigin}
\end{subfigure}
\caption{Evolution of the expectation of the distances of $v_f$ and $x_*$ over time.}
\label{fig:evolutionNormAlpha}
\end{figure}

These results indicate that the weight parameter $\alpha$ has more influence on the outcome of the simulation than the number of particles $N$. Even for small $N$ we obtain good results in dimension $d=20$ if $\alpha$ is sufficiently large. Note further, that the function evaluations correlate with $N$, thus small $N$ implies few function evaluations. In many applications function evaluations are very costly and optimization algorithms may be applied only if the number is function evaluations is small. The numerical results therefore indicate that the CBO algorithm might be a good choice for the global optimization of objective functions that have large evaluation costs.

\section{Conclusion}\label{sec:conclusions}
We presented a new swarm intelligence model, namely the CBO algorithm, in the spirit of consensus formation, which can be used for global optimization purposes. As already mentioned in the introduction, most global optimization approaches have their specific advantages and drawbacks. Here we emphasize on the analytical investigation of the corresponding mean-field limit, which allowed us to study the convergence behavior and to also draw some conclusions on the microscopic system. This is an initial step in the direction of combining tools from global optimization with mean-field analysis. The results are encouraging for future research and might be applied to other approaches such as PSO, ACO or other well-known metaheuristics.

\section*{Acknowledgment}
We thank Nicola Bellomo and Jos\'e Antonio Carrillo for their constructive comments and suggestions. Oliver Tse gratefully acknowledges financial support from the Nachwuchsring of the University of Kaiserslautern.

\bibliographystyle{plain}
\bibliography{CBO}

\end{document}